\documentclass[a4paper,10pt]{article}
\usepackage{authblk}
\usepackage[english]{babel}
\usepackage[utf8]{inputenc}
\usepackage[T1]{fontenc}
\usepackage{libertine}
\usepackage{hyperref}
\hypersetup{
    colorlinks=true,
    linkcolor=blue,
    filecolor=magenta,      
    urlcolor=cyan,
    pdftitle={Cot moments},
}
\usepackage{graphicx}

\usepackage{amsmath,amssymb,amsthm,amsbsy}
\usepackage{bbm} % for \mathbbm{1}

\usepackage{enumitem} % arabic, alph, roman

\usepackage{geometry}
\usepackage{fancyvrb}

\theoremstyle{plain}
\newtheorem*{remark*}{Remark}

\newtheorem{theorem}{Theorem}[section]
\newtheorem{corollary}[theorem]{Corollary}
\newtheorem{lemma}[theorem]{Lemma}
\newcommand{\N}{\mathbb{N}} %% integers
\newcommand{\R}{\mathbb{R}} %% reals

\newcommand{\bpi}{ {\boldsymbol\pi}}

\newcommand{\calP}{\mathcal{P}}

\newcommand{\eqdef}{\overset{\mathrm{def}}{=}}

\newcommand{\Heve}[2]{{H_0(#1,#2)}}
\newcommand{\Hodd}[2]{{H_1(#1,#2)}}
\newcommand{\Reve}[1]{{R_0(#1)}}
\newcommand{\Rodd}[1]{{R_1(#1)}}
\newcommand{\Seve}[1]{{S_0(#1)}}
\newcommand{\Sodd}[1]{{S_1(#1)}}
\newcommand{\Ceve}[1]{{C_0(#1)}}
\newcommand{\Codd}[1]{{C_1(#1)}}

\providecommand{\keywords}[1]
{
  \small	
  \textbf{\textit{Keywords---}} #1
}

\begin{document}
\title{Moments of the $\cot$ function, central factorial numbers and their links with the
Dirichlet eta function at odd integers}
\author{Serge Iovleff\thanks{serge.iovleff@utbm.fr}}
%{} 
\affil{University of Technology of Belfort-Montbéliard\\
P\^ole Énergie-Informatique,\\
13 rue Ernest Thierry-Mieg,\\
90010 Belfort cedex, France
}
% \affil{
% Laboratoire Paul Painlev\'e, \\
% CNRS U.M.R. 8524,\\
% 59 655 Villeneuve d'Ascq Cedex, France
% }
% \affil{Inria, Team Modal, Parc scientifique de la Haute-Borne\\
% 40, avenue Halley – Bât A – Park Plaza\\
% 59650 Villeneuve d’Ascq – France}

\maketitle
\begin{abstract}
We investigate the properties of the moments of the $\cot$ function using
the central factorial numbers. Using a new integral representation of the central factorial numbers,
we find a new way to express these moments in terms of recursive sums and integrals.
This allows us to compute 'recursive' generalized harmonic series and multiple
integrals as a linear combination of the Dirichlet eta functions at odd integers.
\end{abstract}

\section{Introduction}\label{sec:introduction}

\paragraph{The problem} Let $C(m)$ denote the $m$th moment of the $\cot$ function, we define by the formula
\begin{equation}\label{eq:moment-int-cot}
C(m) \eqdef \frac{1}{m!}
\int_0^\pi \frac{\theta^{m}}{2} \cot\left(\frac{\theta}{2}\right) d\theta, \quad m=1,2,\ldots
\end{equation}
There are few results
about these moments. The reference text by Prudnikov et al.
(\cite{prudnikov1986integrals}, pp. 436) gives that 
$$
C(m)=\frac{\pi^n}{m!} \left[ \frac{1}{m} - 2 \sum_{k=1}^\infty \frac{\zeta(2k)}{4^k(n+2k} \right]
$$
which we not find very useful and informative. Alternatively, using the Riemann-Lebesgue lemma,
one has
$$
C(m) = \frac{2}{m!} \sum_{n=1}^\infty \int_0^{\pi} \theta^m \sin(n\theta) d\theta.
$$
The integrals involving $\theta^m\sin(n\theta)$ can be evaluated for any pair
$(m,n)\in\N^2$ (see \cite{gradshteyn2014}, 3.761-5). The resulting series
give a linear combination of the Dirichlet eta function and,
when $m$ is even, the zeta function. More specifically, we have
\begin{equation}\label{eq:Cm-Series}
C(m) =  \sum_{l=0}^{\left\lfloor \frac{m}{2}\right\rfloor}  (-1)^l \frac{\pi^{m-2l}}{(m-2l)!} \eta(2l+1)
+  \frac{1 + (-1)^m}{2}
{(-1)^{\left\lfloor \frac{m}{2}\right\rfloor}}\zeta(m+1).    
\end{equation}
The extra term appearing when $m$ is even seems to indicate that there exists
a different behavior of the moments of the function $\cot$ for the odd and even moments.
This is what will highlight
the results presented hereafter (theorems \ref{th:main-theorem}, \ref{th:main-theorem-integrals},\ref{th:main-theorem-series}).

Our approach in order to evaluate the moments of the $\cot$ function is very different from
the one used above:
we first observe that by using an elementary change of variable, it is possible to make
the generating function of the cfns appear inside the integral (\ref{eq:moment-int-cot}).
Integrating this generating function term by term and applying a novel integral representation
of the cfns (theorem \ref{th:Hkj-integral}), we arrive at a new series representation
of these moments.

By comparing these two series, we can then evaluate a family of integrals and series that
do not seem obvious to evaluate directly. Finally, we point out that (maybe) the most
interesting results that may be found in this paper are the proof of the intermediary results.
In fact, in Section \ref{sec-recursive-series-evaluation}, we evaluate exactly 'recursive'
generalized harmonic series using a method that seems quite new, and we did not find elsewhere.

% Finally, we point out that D. Cvijovi and al. (\cite{CVIJOVIC2009244}) gave some generalizations
% when the integration interval is of the form $[0,\ p\pi/q]$ with $p$ and $q$ integers.

\paragraph{Main results}
Before stating our main results and some of their consequences, we recall
an elementary identity that we will use throughout this paper.
For any $j>-1$ and $l=0,1,\ldots$, we have
\begin{equation}\label{eq:Integral-power-Log}
\frac{1}{(j+1)^{l+1}} = \frac{(-1)^l}{l!} \int_0^1 x^{j} \log^l(x) dx.
\end{equation}

Following the convention used throughout this paper, we denote by
$\Ceve{k}$ the
even moments of the $\cot$ function (when $m=2k$, $k>0$) and by $\Codd{k}$
the odd moments of the $\cot$ function (when $m=2k+1$, $k\geq0$). Our first result gives an integral representation for the moments of the $\cot$ function
in terms of two kernel functions
\begin{theorem}\label{th:main-theorem}
For any $k\geq 0$, we have
\begin{eqnarray}
\Codd{k} &=& \sum_{l=0}^k 2^{2l+1} \frac{\pi^{2k-2l}}{(2k-2l)!}
\int_{C_{l}}  K_1(x_1\ldots x_{l})
\prod_{i=1}^{l} \frac{\log(x_i)dx_i}{1-\prod_{m=1}^i x_m^2},\label{eq:C1kwithK1}\\
\Ceve{k+1} &=& \sum_{l=0}^k \frac{\pi^{2k-2l}}{(2k-2l+1)!}
\int_{C_{l}}  K_0(x_1 \ldots x_{l})
\prod_{i=1}^{l} \frac{\log(x_i)dx_i}{1-\prod_{m=1}^i x_m }\label{eq:C0kwithK0}
\end{eqnarray}
with $K_1$ and $K_0$ given by the entire series
\begin{equation*}
K_1(z) = \sum_{j=0}^\infty \binom{2j}{j} \frac{1}{(2j+1)^2}{ \left(\frac{z}2\right)^{2j}}\qquad\text{and}\qquad
K_0(z) = \frac{1}{2} \sum_{j=1}^\infty { \frac{\left(4z\right)^{j}}{j^3\binom{2j}{j}}},
\end{equation*}
% \begin{eqnarray*}
% K_1(z) & = & \sum_{j=0}^\infty \binom{2j}{j} \frac{1}{(2j+1)^2}
% \left(\frac{z}2\right)^{2j} \\
% K_0(z) & = &  \sum_{j=0}^\infty \frac{\left(4z\right)^{j}}{j^3\binom{2j}{j}}.
% \end{eqnarray*}
$C_r$ denoting the unit square in $\R^r$ and using the convention that the integrals are equal
to one when $r=0$.
\end{theorem}
It is not difficult to verify that the kernels $K_1$ and $K_0$ have an explicit form
(see \cite{Lehmer1985}  for the second one) given by
$$
K_1(z) = \frac1z \int_0^z \frac{\arcsin(y)}{y} dy\qquad\text{and}\qquad
K_0(z) = 2 \int_0^z \frac{\arcsin(\sqrt{y})^2 }{y} dy.
$$
Based on this theorem, we can easily deduce the following theorem
\begin{theorem}\label{th:main-theorem-integrals}
For any $k\geq 0$, we have
\begin{eqnarray}
\Codd{k} &=& -\sum_{l=0}^k 2^{2l+1} \frac{\pi^{2k-2l}}{(2k-2l)!}
\int_{C_{l+1}}
\frac{\log(x_0) dx_0}{ \sqrt{1-\prod_{i=0}^{k-l} } x_{i}^{2} }
\prod_{i=1}^{k-l}\frac{ \log(x_i) dx_i}{1- \prod_{m=1}^i x_{m}^2}
     \label{eq:C1_integral} \\
\Ceve{k+1} &=&  \sum_{l=0}^k \frac{\pi^{2k-2l}}{(2k-2l+1)!}
\int_{C_{l+1}}
{\arcsin^2\left(\sqrt{\prod_{i=0}^{l}x_i }\right) }
\frac{dx_0}{x_0}
 \prod_{i=1}^{l} \left(\frac{ \log(x_i)} {1- \prod_{m=1}^i x_{m} }
 \frac{ dx_i} {x_i}\right)
 \label{eq:C0-integral} 
\end{eqnarray}
\end{theorem}
\begin{proof}
We recall that for any $|x|\leq1$ we have
(see \cite{stanley2001enumerative} example 1.1.15 for the first identity and \cite{Lehmer1985} for the second)
$$
\sum_{j=0}^\infty \binom{2j}{j} \left( \frac{x}{2}\right)^{2j}=\frac{1}{\sqrt{1-x^2}}
\qquad\text{and}\qquad
\frac{1}{2} \sum_{j=1}^\infty \frac{(2x)^{2j}}{j^2\binom{2j}{j}} = \arcsin^2{\left({x} \right)}.
$$
Using the identity (\ref{eq:Integral-power-Log}) with $j=2$ and $l=1$, we get $1/(2j+1)^2 = - \int_0^1 x_0^{2j}\log(x_0)dx_0$ which gives us the first equality. For the second equality, write
$1/j=\int_0^1 x_0^{j-1}dx_0$ and conclude.
\end{proof}

Finally, we can also express theorem \ref{th:main-theorem} as a sum of weighted recursive series
\begin{theorem}\label{th:main-theorem-series}
For any $k\geq 0$, we have
\begin{equation*}
\label{eq:C0-series} 
\Codd{k} = \sum_{l=0}^k (-1)^l \frac{2^{2l+1} \pi^{2k-2l}}{(2k-2l)!} \Sodd{l}
\quad\mbox{ and }\quad
\Ceve{k+1} = \sum_{l=0}^k  (-1)^{l}\frac{\pi^{2l}}{(2k-2l+1)!}  \Seve{l} 
\end{equation*}
with
\begin{eqnarray*}
\Sodd{k} &=& \sum_{j_0=0}^\infty \binom{2j_0}{j_0}
\left(\frac{1}{2}\right)^{2j_0}\frac{1}{(2j_0+1)^2}
\sum_{j_1=j_0}^\infty \sum_{j_2=j_1}^\infty \ldots \sum_{j_k=j_{k-1} }^\infty 
\frac{1}{(2j_1+1)^2\ldots (2j_k+1)^2}\\
\Seve{k} &=&\frac{1}{2} \sum_{j_0=1}^\infty
\frac{2^{2j_0}}{\binom{2j_0}{j_0}}\frac{1}{j_0^3}
\sum_{j_1=j_0}^\infty \sum_{j_2=j_1}^\infty \ldots \sum_{j_k=j_{k-1} }^\infty 
\frac{1}{j_1^2j_2^2 \ldots j_k^2}
\end{eqnarray*}
\end{theorem}
\begin{proof}
In equation \ref{eq:C1kwithK1}, using a development in series of $1/(1-x_1^2\ldots x_l^2)$, we get
\begin{multline*}
\int_{C_{l}}  K_1(x_1\ldots x_{l}) \prod_{i=1}^{l} \frac{\log(x_i)dx_i}{1-\prod_{m=1}^i x_m^2}= \\
\int_{C_{l}} \sum_{j_0=0}^\infty \binom{2j_0}{j_0} \left(\frac{1}{2} \right)^{2j_0} \frac{1}{(2j_0+1)^2} 
\left( \sum_{j_1=0}^\infty (x_1x_2\ldots x_{l})^{2j_0+2j_1}\right)\log(x_l) dx_l
\prod_{i=1}^{l-1} \frac{\log(x_i)dx_i}{1-\prod_{m=1}^i x_m^2} 
\end{multline*}
Integrating with respect to $x_l$ (identity  \ref{eq:Integral-power-Log}), we get
\begin{multline*}
\int_{C_{l}}  K_1(x_1\ldots x_{l}) \prod_{i=1}^{l} \frac{\log(x_i)dx_i}{1-\prod_{m=1}^i x_m^2}= \\
\int_{C_{l}} \sum_{j_0=0}^\infty \binom{2j_0}{j_0} \left(\frac{1}{2} \right)^{2j_0} \frac{1}{(2j_0+1)^2} 
\left( \sum_{j_1=j_0}^\infty \frac{-1}{(2j+1)^2}(x_1x_2\ldots x_{l-1})^{2j_1}
\prod_{i=1}^{l-1} \frac{\log(x_i)dx_i}{1-\prod_{m=1}^i x_m^2} \right).
\end{multline*}
Developing in series $1/(1-x_1^2\ldots x_{l-1}^2)$ and integrating with respect to $x_{l-1}$ and so on
gives the announced result.

The proof of the second equality is similar and has been omitted.
\end{proof}

\paragraph{Some consequences}
Setting $k=0$ and $k=1$ and comparing the formula (\ref{eq:Cm-Series})
with the representation given in Theorem \ref{th:main-theorem-integrals} and
\ref{th:main-theorem-series}, we obtain the following identities using the odd moments
$$
\int_0^1 \frac{-\log(x_0)}{\sqrt{1-x_0^2}}dx_0
=\sum_{j_0=0}^\infty \binom{2j_0}{j_0} \left(\frac{1}{2}\right)^{2j_0}\frac{1}{(2j_0+1)^2}
= \frac{\pi}{2}\log(2),
$$
$$\int_0^1\int_0^1 \frac{\log(x_0)\log(x_1)}{\sqrt{1-x_0^2x_1^2}(1-x_1^2)}dx_0dx_1
=\sum_{j_0=0}^\infty \binom{2j_0}{j_0} \left(\frac{1}{2}\right)^{2j_0}\frac{1}{(2j_0+1)^2}
\sum_{j1=j0}^\infty \frac{1}{(2j_1+1)^2}
=\frac{\pi^3}{24}\log(2)+\frac{\pi}{8}\eta(3).
$$
Using the even moments, we get
$$
\int_0^1 \frac{\arcsin^2(\sqrt{x_0})}{x_0}dx_0
= \frac{1}{2} \sum_{j_0=1}^\infty
\frac{2^{2j_0}}{\binom{2j_0}{j_0}}\frac{1}{j_0^3}
=\frac{\pi^2}{2}\log(2)-\frac{7}{3}\eta(3)
$$
$$
\int_0^1\int_0^1 \frac{\arcsin^2(\sqrt{x_0x_1})}{x_0x_1}\frac{\log(x_1)}{1-x_1}dx_0dx_1
= \frac{-1}{2} \sum_{j_0=1}^\infty
\frac{2^{2j_0}}{\binom{2j_0}{j_0}}\frac{1}{j_0^3} \sum_{j_1=j_0}^\infty \frac{1}{j_1^2}
=-\frac{\pi^4}{24}\log(2)-\frac{\pi^2}{9}\eta(3)+\frac{31}{15}\eta(5).
$$
It is possible to find similar (and more complicated) expressions with $k=2$, etc.,
but it does not give a general formula for these integrals/recursive sums.

\section{Evaluation of recursive series involving power of integers}
\label{sec-recursive-series-evaluation}
In this section, we are proving some intermediate results that are interesting in their own right. In part \ref{subsec:series-S0-S1}, we evaluate exactly the recursive generalized harmonic series
\begin{eqnarray*}
\Rodd{k} &\eqdef& \sum_{i_1=0}^\infty \frac{1}{(2i_1+1)^2}\sum_{i_2=0}^{i_1} \frac{1}{(2i_2+1)^2}\ldots
\sum_{i_k=0}^{i_{k-1}} \frac{1}{(2i_k+1)^2}
\\
\Reve{k} &\eqdef& \sum_{i_1=1}^\infty \frac{1}{i_1^2}\sum_{i_2=1}^{i_1} \frac{1}{i_2^2} \ldots
\sum_{i_k=1}^{i_{k-1}} \frac{1}{i_k^2}.   
\end{eqnarray*}
and in part \ref{subsec:series-A0-A1} the “dual” series
\begin{eqnarray}
A_1(k) &\eqdef& \sum_{i_1=0}^\infty \frac{1}{(2i_1+1)^2}
\sum_{i_2=i_1+1}^\infty\frac{1}{(2i_2+1)^2} \ldots  \sum_{i_k=i_{k-1}+1}^\infty \frac{1}{(2i_{k}+1)^2}
\label{def:A1}\\
A_0(k) &\eqdef& \sum_{i_1=1}^\infty \frac{1}{i_1^2}
\sum_{i_2=i_1+1}^\infty\frac{1}{i_2^2} \ldots  
\sum_{i_k=i_{k-1}+1}^\infty \frac{1}{i_{k}^2},\label{def:A0}
\end{eqnarray}
with the convention $A_1(0) = A_0(0)=1$.

\subsection{Evaluation of the series $\Rodd{k}$ and $\Reve{k}$}
\label{subsec:series-S0-S1}
In order to prove the next results, we use a result of (Vella, \cite{Vella2008})
which is an application  of the Faá di Bruno (1855) result. We believe that
this usage is new.

\begin{theorem}\label{th:S1-Evaluation}
For any $k>0$, we have
$$
\Rodd{k} = \left(\frac{\pi}{2}\right)^{2k} \frac{E_{2k}^\star}{(2k)!}
$$
with $E_{2k}^\star$ denoting the Euler (zig) numbers (\href{https://oeis.org/A000364}{OEIS A000364}, even powers only of $\sec(x) = 1/\cos(x)$).
\end{theorem}
\begin{proof}
Observe first that for any $l\geq 1$, we have
\begin{equation}\label{eq:sum-2i+1-puiss-2l}
\sum_{i=0}^\infty \frac{1}{(2i+1)^{2l}}
=\sum_{i=1}^\infty \frac{1}{i^{2l} }
-\frac{1}{2^{2l} } \sum_{i=1}^\infty \frac{1}{i^{2l} }
= \frac{2^{2l}-1}{2^{2l}} \zeta(2l)
= (2^{2l}-1) \frac{|B_{2l}|\pi^{2l}}{2(2l)!},
\end{equation}
with $B_{n}$ demoting the Bernoulli numbers. For $n\geq 1$, let $x_n=1/(2n-1)^2$ and fix $0<k\leq n$.
The complete homogeneous symmetric polynomials of degree $k$ in $x_1,\ldots,x_n$ is defined as
$$
h_k (x_1, x_2, \dots,x_n) = \sum_{1 \leq i_1 \leq i_2 \leq \cdots \leq i_k \leq n} x_{i_1} x_{i_2} \cdots x_{i_k}.
$$
Let $\calP(k)$ denote the set of all partitions of the integer $k$.
For $\bpi\in\calP(k)$ we write $\bpi=\left[1^{\pi_1},\ldots, k^{\pi_{k}}\right]$
with $\pi_l$ denoting the multiplicities of the part $l$ in the current partition of $k$.
With this convention, we have trivially, for any $\bpi\in\calP(k)$,
$k=\sum_{l=1}^{k} l \pi_l$ and we define the length of the partition
as $l(\bpi)=\sum_{l=1}^{k}\pi_l$. 
Using the cycle index formula (\cite{weisstein}), we have
$$
k! \, h_k(x_1,\ldots,x_n) = \sum_{\bpi\in \calP(k)} a(\bpi) 
                       \prod_{l=1}^{k}  \left(\sum_{j=1}^n x_j^{l} \right)^{\pi_l}
$$
with $a(\bpi)$ denoting the number of permutations whose cycle structure is the given partition. These numbers $a(\bpi)$ have the explicit expression
$$
a(\bpi)= k! \prod_{l=1}^{k} \frac{1}{ \pi_l! l^{\pi_l}},\quad \bpi \in \calP(k).
$$
Taking the limit in $n$,
and using (\ref{eq:sum-2i+1-puiss-2l}), we get
\begin{equation*}
\Rodd{k} = \frac{1}{k!} \sum_{\bpi\in \calP(k)} a(\bpi)\prod_{l=1}^{k}
\left[\frac{2^{2l}-1}{2^{2l}} \zeta(2l)\right]^{\pi_l}.
\end{equation*}
Simplifying, we obtain the following
\begin{equation*}
\Rodd{k} = {\pi^{2k}} \sum_{\bpi\in \calP(k)}  \prod_{l=1}^{k} \frac{1}{\pi_l!} 
\left[ (2^{2l}-1) \frac{|B_{2l}|}{2l(2l)!}\right]^{\pi_l}
= {\pi^{2k}} \sum_{\bpi\in \calP(k)} \binom{l(\bpi)}{\pi_1,\ldots,\pi_k} \frac{1}{l(\bpi)!}
\prod_{l=1}^{k} 
\left[ (2^{2l}-1) \frac{|B_{2l}|}{2l(2l)!}\right]^{\pi_l},
\end{equation*}
with $\binom{n}{\alpha_1,\ldots,\alpha_k}$ denoting the multinomial coefficient.
Using the generating function of the Bernoulli numbers it is easily found that
$$ \frac{z}{2} \cot\left( \frac{z}{2} \right)-z \cot\left(z\right)  = \sum_{n=1}^\infty (2^{2n}-1) |B_{2n}| \frac{z^{2n}}{(2n)!}.$$
Consider the entire series
$$
g(z) = \sum_{n=1}^\infty (2^{2n}-1) |B_{2n}| \frac{z^{2n}}{(2n)(2n)!}.
$$
Derivating with respect to $z$, we get
$$
zg'(z)=\frac{z}{2} \cot\left( \frac{z}{2} \right)-z \cot\left(z \right)
$$
and thus
$$
g(z) = \log\left(\frac{\sin(z/2)}{\sin(z)} \right)+\log(2)= \log\left(\frac{1}{\cos(z/2)} \right)
$$
is the generating function of the sequence $(2^{2n}-1)  \frac{|B_{2n}|}{(2n)(2n)!}$.
We apply Corollary 4 of (Vella, \cite{Vella2008}) with $g$ given above and $f(z)=\exp(z)$. Recognizing that $1/\cos(z/2)$ is the generating function of $E_{2n}^\star/2^n$, we arrive at the conclusion.
\end{proof}
We have a similar theorem for the series $\Reve{k}$.
\begin{theorem}\label{th:S0-Evaluation}
For any $k>0$, we have
$$
\Reve{k} = 2 (2^{2k-1} -1) \frac{|B_{2k}|\pi^{2k}}{(2k)!}.
$$
\end{theorem}

\begin{proof}
The arguments are very similar to the previous one, so we give only a sketch of the proof.
Using the same methodology, we arrive at the expression
\begin{equation*}
\Reve{k} = \frac{1}{k!} \sum_{\bpi\in \calP(k)} a(\bpi)\prod_{l=1}^{k} \zeta(2l)^{\pi_l}.
\end{equation*}
Simplifying, we get the following
\begin{equation*}
\Reve{k} = {(2\pi)^{2k}} \sum_{\bpi\in \calP(k)} \prod_{l=1}^{k} \frac{1}{\pi_l!} \left(\frac{|B_{2l}|}{2l(2l)!} \right)^{\pi_l}
={(2\pi)^{2k}} \sum_{\bpi\in \calP(k)} \binom{l(\bpi)}{\pi_1,\ldots,\pi_k} \frac{1}{l(\bpi)!}
\prod_{l=1}^{k} 
\left[\frac{|B_{2l}|}{2l(2l)!}\right]^{\pi_l}.
\end{equation*}
Next, we find that
$$
\log(z/2) - \log(\sin(z/2)) = \sum_{i=1}^\infty \frac{|B_{2n}|}{2n} \frac{z^{2n}}{(2n)!}
$$
and that $g(z) = \frac{z}{2} \frac{1}{\sin(z/2})$
is the generating function of the sequence
$$
2(2^{2n-1}-1) \frac{|B_{2n}|}{(2n)!2^{2n}}, \qquad n\geq 1.
$$
See also \href{https://oeis.org/A036280}{OEIS A036280} and \href{https://oeis.org/A036281}{OEIS A036281}.
\end{proof}

\subsection{Evaluation of the series $A_1(k)$ and $A_0(k)$}
\label{subsec:series-A0-A1}

From the previous results, we obtain the corollary
\begin{corollary}
For any $k>0$, the following holds
$$
A_1(k) = \left(\frac{\pi}{2}\right)^{2k} \frac{1}{(2k)!}
\qquad \mbox{ and } \qquad
A_0(k) = \frac{\pi^{2k}}{(2k+1)!}.
$$
\end{corollary}
\begin{proof}
We just need to prove the first claim
as the evaluation of $A_0(k)$ is already proved by \cite{Hoffman1992}
and \cite{Merca2015} using similar arguments to the one used below.

It is easy to find that when $k=1$, $A_1=\pi^2/8$ but the remaining constants
are more difficult to compute. Let thus consider $k\geq2$, with an obvious meaning for
the coefficient $a_i^{(k)}$, we have
\begin{eqnarray*}
A_1(k) &=& \sum_{i_1=0}^\infty \frac{a_{i_1}^{(k)}}{(2i_1+1)^2}
         = \sum_{i_1=0}^\infty \frac{1}{(2i_1+1)^2} \left( A_{k-1} - \sum_{i_2=0}^{i_1}  \frac{a_{i_2}^{(k-1)}}{(2i_2+1)^2} \right) \\
       &=& \Rodd{1} A_1(k-1) - \left[\sum_{i_1=0}^\infty \frac{1}{(2i_1+1)^2} 
           \sum_{i_2=1}^{i_1} \frac{a_{i_2}^{(k-1)}}{(2i_2+1)^2} \right] \\
       &=& \ldots \\
       &=& \sum_{l=1}^k (-1)^{l+1} \Rodd{l} A_1(k-l).
\end{eqnarray*}
with $\Rodd{l}$ defined in theorem \ref{th:S0-Evaluation}. The result is proven by recurrence.
Assuming the expression of $A_1(l)$ is true for $1\leq l\leq k-1$, it will be true
for $l=k$ if
$$\sum_{l=1}^k (-1)^{l+1} \binom{2k}{2l} E_{2l}^\star = 1 \qquad \Longleftrightarrow 
\qquad \sum_{l=0}^k (-1)^{l+1} \binom{2k}{2l} E_{2l}^\star = 0.
$$
This result can be found in \cite{simmons2007calculus} and on the web site
\href{https://proofwiki.org/wiki/Sum_of_Euler_Numbers_by_Binomial_Coefficients_Vanishes}{proofwiki}
(Sum of Euler Numbers by Binomial Coefficients Vanishes).

\end{proof}

\subsection{Integral representation of $A_1(k)$ and $A_0(k)$}
The next lemma gives an integral representation of the series $A_0(k)$ and $A_1(k)$
\begin{lemma}\label{lemma:Ak-evaluation}
For any $k\geq 0$, the following equality holds
$$
A_1(k) = (-1)^{k} \int_{C_k}\prod_{i=1}^{k} 
                  \frac{  x_i^{2(i-1)}  \log(x_i) dx_i}{ \left(1-\prod_{m=i}^k x_m^2\right) }
\quad \mbox{ and }\quad
A_0(k) = (-1)^{k} \int_{C_k} \prod_{i=1}^{k}
           \frac{ x_i^{i-1}  \log(x_i) dx_i }{ \left(1-\prod_{m=i}^k x_m\right) }
$$
with $A_0(k)$ and $A_1(k)$ defined in (\ref{def:A0}) and (\ref{def:A1}).
\end{lemma}
\begin{proof}
We prove the result for $A_1(k)$. The proof for $A_0(k)$ is similar and is omitted.
We have to show that the integral on the right hand side is equal to
$$
\sum_{i_1=0}^\infty \frac{1}{(2i_1+1)^2}
\sum_{i_2=i_1+1}^\infty\frac{1}{(2i_1+1)^2} \ldots  \sum_{i_l=i_{l-1}+1}^\infty \frac{1}{(2i_{k}+1)^2}.
$$
Developing in series the (unique) term with $x_1^2$ we obtain
$$
\frac{1}{1- x_1^2x_2^2\ldots x_k^2} = \sum_{i_1=0}^\infty (x_1^2x_2^2\ldots x_k^2)^{i_1}
$$
and integrating with respect to $\log(x_1)\, dx_1$ we get that the right hand side is equal to
$$
\sum_{i_1=0}^{\infty} \frac{1}{(2i_1+1)^2} \int_{C_{k-1} }\frac{ (-1)^{k-1} \prod_{i=2}^{k} x_i^{2(i-1)+2i_1}  \log(x_i) dx_i}
{ \prod_{i=2}^{l} \left(1-\prod_{m=i}^k x_m^2\right) }.
$$
Iterating the process with $x_2$, $x_3$,...,$x_k$ we arrive at the conclusion.
\end{proof}

\section{Integral representation of the central factorial numbers (cfns)}
\label{sec:results-cfn}

In this part, we define the (unsigned) central factorial numbers and give
an integral representation of the "multinomial" part.
The sequences $t_0(k,n)=|t(2n,2k)|$ (\href{https://oeis.org/A008955}{OEIS A008955})
and $t_1(k,n)=|t(2n+1,2k+1)|$ (\href{https://oeis.org/A008956}{OEIS A008956}) used
in this paper are obtained from the recurrences
$$
t_0(k,n) = t_0(k-1,n-1)+(n-1)^2 t_0(k,n-1),\qquad
t_1(k,n) = t_1(k-1,n-1)+\left(n-\frac{1}{2}\right)^2 t_1(k,n-1)
$$
with diagonal and lower part given by
$$
t_0(n,n)=t_1(n,n) = 1, \mbox{ for all } n\in\N, \qquad t_0(k,n)=t_1(k,n)=0, \mbox{ if } k<n,$$
and row of index zero given by
$$
t_0(0,n) = 0, \qquad t_1(0,n)=\frac{(2n+1)!!^2}{4^n},\qquad \mbox{ if } n>0.
$$
The first cfns for $0\leq k\leq5$ and $0\leq n\leq5$ are given in table \ref{tab:cfn}.

\begin{table}[htb]
\centering
$$
\displaystyle
\begin{array}{l|ccccccc}
k\backslash n & 0& 1 & 2 & 3 & 4 & 5 & \dots\\
\hline
0& 1& 0 & 0 & 0 & 0 & 0 &\dots \\
1& 0&1 & 1 & 4 & 36 & 576& \dots \\
2& 0&0 & 1 & 5 & 49 & 820& \dots \\
3& 0&0 & 0 & 1 & 14 & 273& \dots \\
4& 0&0 & 0 & 0 & 1 & 30& \dots \\
5& 0&0 & 0 & 0 & 0 & 1& \dots \\
\vdots &\vdots &\vdots&\vdots&\vdots&\vdots&\vdots&\ddots
\end{array}
\qquad\displaystyle
\begin{array}{l|ccccccc}
k\backslash n & 0& 1           & 2            & 3              & 4 & 5 & \dots\\
\hline
0             & 1 & \frac{1}{4} & \frac{9}{16} & \frac{225}{64} & \frac{11025}{256} & \frac{893025}{1024}& \dots \\
1 & 0 & 1 & \frac{5}{2} & \frac{259}{16} & \frac{3229}{16} & \frac{1057221}{256}& \dots \\
2 & 0 & 0 & 1 & \frac{35}{4} & \frac{987}{8} & \frac{86405}{32}& \dots \\
3 & 0 & 0 & 0 & 1 & 21 & \frac{4389}{8}& \dots \\
4 & 0 & 0 & 0 & 0 & 1 & \frac{165}{4}& \dots \\
5 & 0 & 0 & 0 & 0 & 0 & 1& \dots \\
\vdots &\vdots &\vdots&\vdots&\vdots&\vdots&\vdots&\ddots
\end{array}
$$
\caption{Arrays of the central factorial number. Left: the numbers $t_0$. Right: the numbers $t_1$}
\label{tab:cfn}
\end{table}
These sequences are related to the central factorial numbers (cfn)
$t(2n,2k)$ and $t(2n+1,2k+1)$ presented and studied in (Butzer and al. \cite{Butzer1989} ).
Note that, with the convention used in this paper, cfns $t_0$ and $t_1$ are transposed
and unsigned with respect to the sequences studied in the cited article. Using the previous
reference, we have also that the cfns $|t(n,m)|$ have the following generating function
(\cite{Butzer1989}, Theorem 4.1.2) for $|z|<2$ and any $m\in\N$
\begin{equation}\label{eq:gen-funct-cfn}
\frac{1}{m!} \left[2\;\mathrm{arcsin}\left(\frac{z}{2}\right)\right]^{m}
=\sum_{n=m}^\infty  |t(n,m)|  \frac{z^{n}}{n!}.
\end{equation}
In particular, if $m=2k$ and $m=2k+1$, we have
\begin{eqnarray}
\frac{1}{(2k)!} \left[2\;\mathrm{arcsin}\left(\frac{z}{2}\right)\right]^{2k}
&=&\sum_{n=k}^\infty  t_0(k,n)  \frac{z^{2n}}{(2n)!}, \quad k\geq 1\label{eq:gen-funct-cfn0}\\
\frac{1}{(2k+1)!} \left[2\;\mathrm{arcsin}\left(\frac{z}{2}\right)\right]^{2k+1}
&=&\sum_{n=k}^\infty  t_1(k,n)  \frac{z^{2n+1}}{(2n+1)!}, \quad k\geq0. \label{eq:gen-funct-cfn1}
\end{eqnarray}

Finally, we define the following recursive harmonic numbers of order two as
$\Heve{k}{n}$ and $\Hodd{k}{n}$ by the initial conditions
$\Heve{0}{0}=1$, $\Heve{0}{n}=0$ if $n\geq 1$,
$\Heve{1}{n} = \Hodd{0}{n}=1$ for all  $n>0$,
and, if $k>1$, the recurrences
\begin{eqnarray*}
\Heve{k}{n} &=& \sum_{i_k=k-1}^{n-1} \frac{1}{i_k^2} \Heve{k-1}{i_k}
= \sum_{i_k=k-1}^{n-1} \frac{1}{i_k^2}
\left(\sum_{i_{k-1}=k-2}^{i_k-1} \frac{1}{i_{k-1}^2} \ldots
\sum_{i_1=1}^{i_2-1} \frac{1}{i_1^2} \right)\\
\Hodd{k}{n} &= &\sum_{i_k=k-1}^{n-1} \frac{1}{(2i_k+1)^2} \Hodd{k-1}{i_k}
= \sum_{i_k=k-1}^{n-1} \frac{1}{(2i_k+1)^2}
\left(\sum_{i_{k-1}=k-2}^{i_k-1} \frac{1}{(2i_{k-1}+1)^2} \ldots
\sum_{i_1=0}^{i_2-1} \frac{1}{(2i_1+1)^2} \right).
\end{eqnarray*}
The first numbers for $0\leq k\leq3$ and $0\leq n\leq5$ are given in table \ref{tab:Hkn}.
\begin{table}[htb]
\centering
$$
\displaystyle
\begin{array}{l|ccccccc}
k\backslash n & 0& 1           & 2            & 3              & 4 & 5 & \dots\\
\hline
0&1 & 0 & 0 & 0 & 0 & 0&\dots\\
1&0 & 1 & 1 & 1 & 1 & 1&\dots\\
2&0 & 0 & 1 & \frac{5}{4} & \frac{49}{36} & \frac{205}{144}&\dots\\
3&0 & 0 & 0 & \frac{1}{4} & \frac{7}{18} & \frac{91}{192}&\dots\\
\vdots &\vdots &\vdots&\vdots&\vdots&\vdots&\vdots&\ddots
\end{array}
\qquad\displaystyle
\begin{array}{l|ccccccc}
k\backslash n & 0 & 1           & 2            & 3              & 4 & 5 & \dots\\
\hline
0&1 & 1 & 1 & 1 & 1 & 1&\dots\\
1&0 & 1 & \frac{10}{9} & \frac{259}{225} & \frac{12916}{11025} & \frac{117469}{99225}&\dots\\
2&0 & 0 & \frac{1}{9} & \frac{7}{45} & \frac{94}{525} & \frac{34562}{178605}&\dots\\
3&0 & 0 & 0 & \frac{1}{225} & \frac{4}{525} & \frac{418}{42525}&\dots\\
\vdots &\vdots &\vdots&\vdots&\vdots&\vdots&\vdots&\ddots
\end{array}
$$
\caption{Arrays of the recursive harmonic numbers of order two. Left: the numbers $H_0$.
Right: the numbers $H_1$}
\label{tab:Hkn}
\end{table}
Finally, for all $k\in \N$ (\cite{Butzer1989}, p. 430)
$$
t_0({k,n}) = (n-1)!^2 \Heve{k}{n},\ n\geq1,
\qquad \mbox{ and }\qquad
t_1({k,n}) = 2^{2k} \binom{2n}{n}\frac{(2n)!}{2^{4n}}\Hodd{k}{n},\ n\geq 0.
$$

The next theorem gives an integral representation of the coefficients
$\Hodd{k}{j}$ and $\Heve{k}{j}$.
\begin{theorem}\label{th:Hkj-integral}
Let $k\geq0$, for any $j \geq 0$, the coefficients $\Heve{k+1}{j}$ and
$\Hodd{k}{j}$ can be expressed as the sum of integrals of the form
\begin{eqnarray}
\Hodd{k}{j} &=& \sum_{l=0}^k \frac{1}{(2l)!} \left(\frac{\pi}{2}\right)^{2l}
\int_{C_{k-l}}
\prod_{i=l+1}^k\frac{   x_{i}^{2j} \log(x_i) dx_i}
     {\left(1- \prod_{m=l+1}^i x_{m}^2\right) }\label{eq:Tood_integral} \\
\Heve{k+1}{j} &=& \sum_{l=0}^k \frac{\pi^{2l}}{(2l+1)!}
\int_{C_{k-l}}
 \prod_{i=l+1}^k \frac{  x_{i}^{j-1} \log(x_i) dx_i}
     { \left(1- \prod_{m=l+1}^i x_{m}\right) }   \label{th:teve-integral} 
\end{eqnarray}
for $l=0,\ldots, k$, with $C_l$ denoting the unit cube in $\R^l$ and with
the convention that the integrals in the right-hand sides are equal to 1 when $l=k$.
\end{theorem}
\begin{proof}
We first show that for $j \geq k$, the coefficients $\Hodd{k}{j}$ can be expressed as integral of the form
\begin{equation}\label{eq:Tkj_integral-1}
\Hodd{k}{j} = (-1)^k \int_0^1\ldots \int_0^1
\sum_{l=0}^k r_l^{(k)}(x_1,\ldots x_k) x_{l+1}^{2j}\ldots x_k^{2j}\,
\log(x_1)\ldots\log(x_k)\, dx_1\ldots dx_k
\end{equation}
with
\begin{equation}
\label{eq:rkk}
r_l^{(k)}(x_1,\ldots x_k) = \frac{ (-1)^{k+l} \prod_{i=1}^{l} x_i^{2(i-1)} }
{\prod_{i=1}^{l} \left(1-\prod_{m=i}^l x_m^2\right) \prod_{i=l+1}^{k} \left(1- \prod_{m=l+1}^i x_{m}^2\right) },
\end{equation}
independent of $j$, for $l=0,\ldots, k$. %, and with the convention $\prod_{m=1}^0 x_m = 1$.
Using repeatedly the identity (\ref{eq:Integral-power-Log}) with $l=1$, we obtain
\begin{equation*}
\Hodd{k}{j} = (-1)^{k} \int_0^1\ldots \int_{0}^1 \left( \sum_{i_k=k-1}^{j-1} x_k^2 \left(\sum_{i_{k-1}=k-2}^{i_k-1} x_{k-1}^2 \ldots \sum_{i_1=0}^{i_2-1} x_1^2 \right) \right) \prod_{j=1}^{k} \log(x_j)dx_j.
\end{equation*}
If $|z|<1$, then
$$
\sum_{i_l = l-1}^{i_{l+1}-1} z^{2i_l} = \frac{z^{2l-2}-z^{2i_{l+1}}}{1-z^2}, \qquad l=1,\ldots,k.
$$
Using this identity, we obtain that the sum inside the integral
(replacing  $j$ by $i_{k+1}$) has the expression
$$
\left(\sum_{i_{k}=k-1}^{i_{k+1}-1} x_{k}^2 \ldots \sum_{i_1=0}^{i_2-1} x_1^2 \right)
= \sum_{l=0}^{k} r_l^{(k)}(x_1,\ldots,x_k) \prod_{m=l+1}^{k} x_m^{2i_{k+1}} 
$$
and that $r_l^{(k)}(x_1,\ldots,x_k)$ obey the following recurrence relations in $k$
\begin{eqnarray*}
r_l^{(k)}(x_1,\ldots,x_k) & = &  r_l^{(k-1)}(x_1,\ldots,x_{k-1}) \frac{-1}{1-(x_{l+1}\ldots x_k)^2}, \quad\mbox{if } 0\leq l<k,\\
r_k^{(k)}(x_1,\ldots,x_k) & = & \sum_{j=0}^{k-1} r_j^{(k-1)}(x_1,\ldots,x_{k-1}) \frac{(x_{j+1}\ldots x_k)^{2k} }{1-(x_{j}\ldots x_k)^2 },
\quad \mbox{otherwise}.
\end{eqnarray*}
with initial value, when $k=1$,
$$r^{(1)}_0(x_1)= \frac{-1}{1-x_1^2}, \qquad r^{(1)}_1(x_1) = \frac{x_1^{0}}{1-x_1^2}. $$
This show that if $0\leq l<k$, we have
$$
r_l^{(k)}(x_1,\ldots,x_k) =  r_l^{(l)}(x_1,\ldots,x_l) \frac{(-1)^{k-l} }
{ (1-x_{l+1}^2) \ldots (1-(x_{l+1}\ldots x_k)^2)}
$$
It remains to verify that $r_l^{(l)}(x_1,\ldots,x_l)$ is given by
\begin{equation}\label{eq:rll}
r_l^{(l)}(x_1,\ldots,x_l) = \frac{ x_1^{0}x_2^2\ldots x_l^{2l-2} } {(1-x_l^2)\ldots(1-x_1^2\ldots x_l^2) }    
\end{equation}
for $1\leq l \leq k$. This can be done by recurrence.
Let us assume that equation (\ref{eq:rll}) is fulfilled for $j=1,\ldots,k-1$.
Then $r_k^{(k)}$ verify the equation
$$
r_{k}^{(k)}(x_1,\ldots,x_k) = \sum_{j=0}^{k-1}(-1)^{k-j} \frac{ r_j^{(j)}(x_1,\ldots,x_j) }
{ (1-x_{j+1}^2) \ldots (1-(x_{j+1}\ldots x_k)^2)},
$$
using the recurrence assumption and summing term by term, we obtain the claim.

Thus, the expression given in equation (\ref{eq:Tkj_integral-1}) can be written
$$
\Hodd{k}{j} = \sum_{l=0}^k 
\int_{C_l} \prod_{i=1}^{l} (-1)^l \frac{ x_i^{2(i-1)}  \log(x_i) dx_i }
                                       { \left(1-\prod_{m=i}^l x_m^2\right) }
\times
\int_{C_{k-l}} \prod_{i=l+1}^k \frac{  x_{i}^{2j} \log(x_i) dx_i }
                                    {  \left(1- \prod_{m=l+1}^i x_{m}^2\right) }
$$
with the convention that when $l=0$,
the integral is equal to one. The first integral does not depend on $j$ and is
equal to $A_1(l)$ thanks to the lemma \ref{lemma:Ak-evaluation}.

Finally, we observe that the polynomials
$r_l^{(k)}(x_1,\ldots x_k) x_{l+1}^{2j}\ldots x_k^{2j}$ are equal to zero if $j<k$,
thus the enunciated result is true for any $j\geq 0$. This ends the claim for $H_1$
and the proof of the first part of the theorem.

The proof for $H_0(k,j)$ is similar (just replace $x_i^2$ by $x_i$) and is omitted.
\end{proof}

\section{Proof of the main results}

Making the change of variable $v=2\sin(\theta/2)$,
$\partial v = \cos(\theta/2) \partial\theta$ in the integral (\ref{eq:moment-int-cot}),
we obtain
$$
C(m) = \int_0^2 \frac{1}{m!} \left[2\arcsin\left(\frac{v}{2}\right)\right]^{m} \frac{dv}{v}.
$$
We recognize in this integral the generating function of the (unsigned) central factorial numbers
(equation \ref{eq:gen-funct-cfn}). Integrating term by term, we get the following representation
of the moment of the $\cot$ function in terms of the cfns
\begin{equation}\label{eq:moment-series-cot}
C(m) = \sum_{n=m}^{\infty} |t(n,m)| \frac{2^n}{n\ n!}
\end{equation}

We get
after simplification the following series representations for the odd and
even moments of the $\cot$ function
\begin{eqnarray*}
\Codd{k} &=& 2^{2k+1}
\sum_{j=0}^\infty \binom{2j}{j}   \frac{1}{4^{j}} \frac{\Hodd{k}{j}}{(2j+1)^2} \\
\Ceve{k} &=&  \frac{1}{2} \sum_{j=1}^\infty \frac{2^{2j}}{j^3\binom{2j}{j}} \Heve{k}{j}.
\end{eqnarray*}
We replace $H_1(k,j)$ and $H_0(k,j)$ in the previous expressions
using Theorem \ref{th:Hkj-integral}. Inverting sum and integral, we arrive to
the identities
\begin{eqnarray*}
\Codd{k} &=& 2^{2k+1} \sum_{l=0}^k \frac{1}{(2l)!}\left(\frac{\pi}{2}\right)^{2l}
\int_{C_{k-l}}  K_1(x_1\ldots x_{k-l})
\prod_{i=1}^{k-l} \frac{\log(x_i)dx_i}{1-\prod_{m=1}^i x_m^2}\\
\Ceve{k} &=& \sum_{l=0}^k \frac{\pi^{2l}}{(2l+1)!}
\int_{C_{k-l}}  K_0(x_1 \ldots x_{k-l})
\prod_{i=1}^{k-l} \frac{\log(x_i)dx_i}{1-\prod_{m=1}^i x_m }.
\end{eqnarray*}
Summing in a reversed order gives the result.

\section{Numerical validations}

All presented results have been thoroughly checked using the Python
symbolic calculus library \verb|sympy| (\cite{SymPy-lib}) and/or the
Python numerical calculus library \verb|scipy| (\cite{SciPy-lib}).
The code (in the form of Jupyter notebooks) is available on demand.

\bibliographystyle{plain}
\bibliography{IntegrateCot}

\end{document}